%% file: 13045.tex
\documentclass[11pt]{article}
\usepackage{latexsym,amssymb,amsmath,amsthm,graphics,graphicx,float,psfrag, epsfig, color, setspace}

\newcommand{\R}{\mathbb{R}}
\newcommand{\C}{\mathbb{C}}

\newcommand{\PP}{\mathbb{P}}
\newcommand{\<}{\langle}
\renewcommand{\>}{\rangle}

\newcommand{\eps}{\varepsilon}
\newcommand{\E}{\mathbb{E}}
\newcommand{\Eic}{E_i^c}
\newcommand{\I}{\mathbf{I}}
\newcommand{\Y}{\mathbf{Y}}
\newcommand{\Ybar}{\bar{\mathbf{Y}}}
\newcommand{\HTp}{\mathbf{H}_{T^\perp}}
\newcommand{\ITp}{\mathbf{I}_{T^\perp}}
\newcommand{\YTp}{\mathbf{Y}_{T^\perp}}
\newcommand{\YT}{\mathbf{Y}_{T}}
\newcommand{\YbarT}{\Ybar_{T}}
\newcommand{\YbarTp}{\Ybar_{T^\perp}}
\newcommand{\XTp}{\mathbf{X}_{T^\perp}}
\newcommand{\HT}{\mathbf{H}_{T}}
\newcommand{\XT}{\mathbf{X}_{T}}
\newcommand{\A}{\mathcal{A}}
\newcommand{\X}{\mathbf{X}}

\newcommand{\Xzerobar}{\bar{\X}^{(0)}}

\newcommand{\W}{\mathbf{W}}
\newcommand{\Wzero}{\W^{(0)}}
\newcommand{\Wone}{\W^{(1)}}
\newcommand{\Wtwo}{\W^{(2)}}
\newcommand{\Wzerobar}{\bar{\W}^{(0)}}
\newcommand{\Wonebar}{\bar{\W}^{(1)}}
\newcommand{\Wtwobar}{\bar{\W}^{(2)}}
\newcommand{\Wk}{\W^{(k)}}
\newcommand{\Wkbar}{\bar{\W}^{(k)}}

\newcommand{\bfH}{\mathbf{H}}

\newcommand{\matrixgeq}{\succeq}

\newcommand{\matrixgt}{\succ}

\newcommand{\bfb}{\mathbf{b}}
\newcommand{\bfz}{\mathbf{z}}
\newcommand{\bfu}{\mathbf{u}}
\newcommand{\zz}{\bfz \bfz^*}
\newcommand{\bfx}{\mathbf{x}}
\newcommand{\bfy}{\mathbf{y}}
\newcommand{\ybar}{\bar{\mathbf{y}}}
\newcommand{\bfnu}{\mathbf{\nu}}
\newcommand{\bfe}{\mathbf{e}}
\newcommand{\Sinv}{\mathcal{S}^{-1}}
\newcommand{\Scal}{\mathcal{S}}
\newcommand{\ee}{\bfe_1 \bfe_1^*}
\newcommand{\PTp}{\mathcal{P}_{T^\perp}}
\newcommand{\tr}{\text{Tr}}
\newcommand{\bfv}{\mathbf{v}}
\newcommand{\vv}{\bfv_1 \bfv_1^*}
\newcommand{\oneE}{\mathbf{1}_E}
\newcommand{\Paxb}{ \mathcal{P}_{\A (\X) = \bfb}}
\newcommand{\Ppsd}{ \mathcal{P}_\text{psd}}
\newcommand{\xxstar}{\bfx \bfx^*}
\newcommand{\xnotxnotstar}{\bfx_0 \bfx_0^*}
\newcommand{\zizi}{ \bfz_i \bfz_i^*}
\newcommand{\indpos}{\iota_{\X \matrixgeq 0}}

\newcommand{\indaxb}{\iota_{\A(\X)= \bfb}}

\newtheorem{theorem}{Theorem}
\newtheorem{lemma}{Lemma}

\newtheorem{corollary}[theorem]{Corollary}

\title{Stable optimizationless recovery from phaseless linear measurements}
\author{Laurent Demanet and Paul Hand \\ $\;$ \\ Massachusetts Institute of Technology, Department of Mathematics, \\ 77 Massachusetts Avenue, Cambridge, MA 02139}
\date{August 2012, Revised October 2013}		

\begin{document}
\maketitle

\begin{abstract}
We address the problem of recovering an $n$-vector from $m$ linear measurements lacking sign or phase information. We show that lifting and semidefinite relaxation suffice by themselves for stable recovery in the setting of $m = O(n \log n)$ random sensing vectors, with high probability. The recovery method is optimizationless in the sense that trace minimization in the PhaseLift procedure is unnecessary.  That is, PhaseLift reduces to a feasibility problem.  The optimizationless perspective allows for a Douglas-Rachford numerical algorithm that is unavailable for PhaseLift.  This method exhibits linear convergence with a favorable convergence rate and without any parameter tuning. 
\end{abstract}

{\bf Acknowledgements.} The authors acknowledge generous funding from the National Science Foundation, the Alfred P. Sloan Foundation, TOTAL S.A., and the Air Force Office of Scientific Research.  The authors would also like to thank Xiangxiong Zhang for helpful discussions.


\section{Introduction}

We study the recovery of a vector $\bfx_0 \in \R^n$ or $\C^n$ from the set of phaseless linear measurements
$$
|\<\bfx_0, \bfz_i\>| \text{ for }  i = 1, \ldots, m,
$$
where $\bfz_i \in \R^n$ or $\C^n$ are known random sensing vectors.  Such amplitude-only measurements arise in a variety of imaging applications, such as X-ray crystallography \cite{H1993, M1990,CESV2011}, optics \cite{W1963}, and microscopy \cite{MISE2008}.  We seek stable and efficient methods for finding $\bfx_0$ using as few measurements as possible.  

This recovery problem is difficult because the set of real or complex numbers with a given magnitude is nonconvex.  In the real case, there are $2^m$ possible assignments of sign to the $m$ phaseless measurements.  Hence, exhaustive searching is infeasible.  In the complex case, the situation is even worse, as there are a continuum of phase assignments to consider.  A method based of alternated projections avoids an exhaustive search but does not always converge toward a solution \cite{ F1982, GS1972, GL1984}. 



In \cite{CESV2011, CSV2011, CMP2011}, the authors convexify the problem by lifting it to the space of $n\times n$ matrices, where $\bfx \bfx^*$ is a proxy for the vector $\bfx$.  A key motivation for this lifting is that the nonconvex measurements on vectors become linear measurements on matrices \cite{BBCE2009}. The rank-1 constraint is then relaxed to a trace minimization over the cone of positive semi-definite matrices, as is now standard in matrix completion \cite{RFP2007}. This convex program is called PhaseLift in \cite{CSV2011}, where  
it is shown that $\bfx_0$ can be found robustly in the case of random $\bfz_i$, if $m = O(n \log n)$. 
The matrix minimizer is unique, which in turn determines $\bfx_0$ up to a global phase.    


The contribution of the present paper is to show that trace minimization is unnecessary in this lifting framework for the phaseless recovery problem.  The vector $\bfx_0$ can be recovered robustly by an optimizationless  convex problem: one of finding a positive semi-definite matrix that is consistent with linear measurements.   We prove there is only one such matrix, provided that there are $O(n \log n)$ measurements. In other words,  the phase recovery problem can be solved by intersecting two convex sets, without minimizing an objective. We show empirically that two algorithms converge linearly (exponentially fast) toward the solution.  We remark that these methods are simpler than methods for PhaseLift because they require less or no parameter tuning.  A result subsequent to the posting of this paper has improved the number of required measurements to $O(n)$ by considering an alternative construction of the dual certificate that allows tighter probabilistic bounds \cite{CL2012}.

In \cite{BCE2006}, the authors show that the complex phaseless recovery problem from random measurements is \emph{determined} if $m \geq 4n-2$ (with probability one).  This means that the $\bfx$ satisfying $| \< \bfx, \bfz_i \> | = | \< \bfx_0, \bfz_i \> |$ is unique and equal to $\bfx_0$, regardless of the method used to find it.   A corollary of the analysis in \cite{CSV2011}, and of the present paper, is that  this property is stable under perturbations of the data, provided $m = O(n\log n)$. This determinacy is in contrast to compressed sensing and matrix completion, where a prior (sparsity, low-rank) is used to select a solution of an otherwise underdetermined system of equations. The relaxation of this prior ($\ell_1$ norm, nuclear norm) is then typically shown to determine the same solution. No such prior is needed here; the semi-definite relaxation helps find the solution, not determine it.



The determinacy of the recovery problem over $n\times n$ matrices may be unexpected because there are $n^2$ unknowns and only $O(n \log n)$ measurements.  What compensates for the apparent lack of data is the fact that the matrix we seek has rank one and is thus on the edge of the cone of positive semi-definite matrices.  Most perturbed matrices that are consistent with the measurements cease to remain positive semi-definite. In other words, the positive semi-definite cone $\X \succeq 0$ is ``spiky" around a rank-1 matrix $\X_0$.  That is, with high probability, particular random hyperplanes that contain $\X_0$ and have large enough codimension will have no other intersection with the cone.

The present paper does not advocate for fully abandoning trace minimization in the context of phase retrieval.  The structure of the sensing matrices appears to affect the number of measurements required for recovery.  Consider measurements of the form $\bfx_0^* \Phi \bfx_0$, for some $\Phi$.  Numerical simulations (not shown) suggest that $O(n^2)$ measurements are needed if $\Phi$ is a matrix with Gaussian i.i.d. entries.  On the other hand, it was shown in \cite{RFP2007} that minimization of the nuclear norm constrained by Tr$(\X \Phi) = \bfx_0^* \Phi \bfx_0$ recovers $\bfx_0 \bfx_0^*$ with high probability as soon as $m = O(n \log n)$.   Other numerical observations (not shown) suggest that it is the symmetric, positive semi-definite character of $\Phi$ that allows for optimizationless recovery.

The present paper owes much to \cite{CSV2011}, as our analysis is very similar to theirs. We wish to also reference the papers \cite{Singer2011, WDM2012}, where phase recovery is cast as synchronization problem and solved via
a semi-definite relaxation of max-cut type over the complex torus (i.e., the magnitude information is first factored out.) The idea of lifting and semi-definite relaxation was introduced very successfully for the max-cut problem in \cite{GW1995}. The paper \cite{Singer2011} also introduces a fast and efficient method based on eigenvectors of the graph connection Laplacian for solving the angular synchronization problem.  The performance of this latter method was further studied in \cite{BSS2012}.

\subsection{Problem Statement and Main Result}
 Let $x_0 \in \R^n$ or $\C^n$ be a vector for which we have the $m$ measurements $
| \< \bfx_0, \bfz_i \> | = \sqrt{b_i}$,
for independent sensing vectors $\bfz_i$ distributed uniformly on the unit sphere.  
 We write the phaseless recovery problem for $\bfx_0$ as 
\begin{align}
\text{Find } \bfx \text{ such that } A(\bfx) = \bfb, \label{exact-axb}
\end{align}
where $A:\R^n \to \R^m$ is given by  $A(\bfx)_i = |\< \bfx, \bfz_i \> |^2$, and $A(\bfx_0)=\bfb$.

Problem \eqref{exact-axb} can be convexified by lifting it to a matrix recovery problem.  
Let $\A$ and its adjoint be the linear operators
\begin{alignat*}{6}
&\A: \quad &&\mathcal{H}^{n\times n} & &\to \R^m  & &\A^*: \quad &&\R^m & &\to  \mathcal{H}^{n \times n} \\
&&&\X & &\mapsto \{\bfz_i^* \X \bfz_i \}_{i=1, \ldots, m},\qquad  & & & &\lambda && \mapsto \sum_i \lambda_i \zizi,
\end{alignat*}
where $\mathcal{H}^{n\times n}$ is the space of $n \times n$ Hermitian matrices.  Observe that $ \A(\xxstar) = A(\bfx) $ for all vectors $\bfx$.  Letting $\X_0 =\xnotxnotstar$, we note that $\A(\X_0) = b$.  We emphasize that $\A$ is linear in $\X$ whereas $A$ is nonlinear in $\bfx$.  

The matrix recovery problem we consider is
\begin{align}
\text{Find } \X \matrixgeq 0 \text{ such that } \A(\X) = \bfb. \label{exact-axb-lifted}
\end{align}
Without the positivity constraint, there would be multiple solutions whenever $m < \frac{(n+1)n}{2}$.  We include the constraint in order to allow for recovery in this classically underdetermined regime.

Our main result is that the matrix recovery problem \eqref{exact-axb-lifted} has a unique solution when there are $O(n \log n)$ measurements.
\begin{theorem}
\label{theorem-exact-lifted}
Let $\bfx_0 \in \R^n$ or $\C^n$ and $\X_0 =\xnotxnotstar$.  Let $m \geq c n \log n$ for a sufficiently large $c$.   With high probability, $\X = \X_0$ is the unique solution to $\X \matrixgeq 0$ and $\A(\X) = b$.  This probability is at least $1 -  e^{-\gamma \frac{m}{n}}$, for some  $\gamma >0$.

\end{theorem}

\noindent As a result, the phaseless recovery problem has a unique solution, up to a global phase, with $O(n\log n)$ measurements.  In the real-valued case, the problem is determined up to a minus sign.  
\begin{corollary}
\label{corollary-exact}
Let $\bfx_0 \in \R^n$ or $\C^n$. Let $m \geq c n \log n$ for a sufficiently large $c$.  With high probability, $\{e^{i \phi} \bfx_0 \}$ are the only solutions to $A(\bfx) = \bfb$.  This probability is at least $1 - e^{-\gamma \frac{m}{n}}$, for some $\gamma >0$.
\end{corollary}

\noindent Theorem \ref{theorem-exact-lifted} suggests ways of recovering $\bfx_0$.  If an  $\X \in \{\X \matrixgeq 0\} \cap \{\X \mid \A( \X) = \bfb\}$ can be found, $\bfx_0$ is given by the leading eigenvector of $\X$.  See Section \ref{section-numerical-results} for more details on how to find $\X$.

\subsection{Stability result}
In practical applications, measurements are contaminated by noise.  To show stability of optimizationless recovery, we consider the model $$A(\bfx) + \bfnu = \bfb,$$
where $\bfnu$ corresponds to a noise term with bounded $\ell_2$ norm, $\|\bfnu \|_2 \leq \eps$.  The corresponding noisy variant of \eqref{exact-axb} is
\begin{align}
\text{Find } \bfx \text{ such that } \|A(\bfx) - \bfb \|_2 \leq \eps \|\bfx_0\|_2^2. \label{noisy-axb}
\end{align}
We note that all three terms in \eqref{noisy-axb} scale quadratically in $\bfx$ or $\bfx_0$.  

Problem \eqref{noisy-axb} can be convexified by lifting it to the space of matrices.  The noisy matrix recovery problem is  
\begin{align}
\text{Find } \X \matrixgeq 0 \text{ such that } \| \A(\X) - \bfb\|_2 \leq \eps \|\X_0\|_2. \label{noisy-axb-lifted}
\end{align}
We show that all feasible $\X$ are within an $O(\eps)$ ball of $\X_0$ provided there are $O(n \log n)$ measurements.

\begin{theorem}
\label{theorem-noisy}
Let $\bfx_0 \in \R^n$ or $\C^n$ and $\X_0 =\xnotxnotstar$. Let $m \geq c n \log n$ for a sufficiently large $c$.  With high probability, 
$$\X \matrixgeq 0 \text{ and } \| \A(\X) - \bfb \|_2 \leq \eps \|\X_0\|_2 \Longrightarrow 
\| \X - \X_0\|_2 \leq C \eps \|\X_0\|_2,
$$
for some $C > 0$.  This probability is at least $1 -  e^{-\gamma \frac{m}{n}}$, for some $\gamma>0$.
\end{theorem}

\noindent As a result, the phaseless recovery problem is stable with $O(n\log n)$ measurements.
\begin{corollary}
\label{corollary-noisy}
Let $\bfx_0 \in \R^n$ or $\C^n$.  Let $m \geq c n \log n$ for a sufficiently large $c$.  With high probability, 
$$ \| A(\bfx) - \bfb \|_2 \leq \eps \|x_0\|_2^2 \Longrightarrow  \left\| \bfx - e^{i \phi} \bfx_0 \right \|_2 \leq C \eps \|\bfx_0 \|_2,
$$
for some $\phi \in [0, 2\pi)$, and for  some $C>0$.
This probability is at least $1 -  e^{-\gamma \frac{m}{n}}$, for some  $\gamma > 0$.
\end{corollary} 

\noindent Theorem \ref{theorem-noisy} ensures that numerical methods can be used to find $\X$.  See Section \ref{section-numerical-results} for ways of finding $\X \in \{\X \matrixgeq 0\} \cap \{\A(\X) \approx \bfb\}$.  As the recovered matrix may have large rank,  we approximate $\bfx_0$ with the leading eigenvector of $\X$.

\subsection{Organization of this paper}
In Section \ref{section-main-result}, we prove a lemma containing the central argument for the proof of Theorem \ref{theorem-exact-lifted}.  Its assumptions involve $\ell_1$-isometry properties and the existence of an inexact dual certificate.  Section \ref{section-proof-main-theorem} provides the proof of Theorem \ref{theorem-exact-lifted} in the real-valued case.  It cites \cite{CSV2011} for the $\ell_1$-isometry properties and Section \ref{section-inexact-dual-certificate} for existence of an inexact dual certificate. In Section \ref{section-inexact-dual-certificate} we construct an inexact dual certificate and show that it satisfies the required properties in the real-valued case.  In section \ref{section-stability} we prove Theorem \ref{theorem-noisy} on stability in the real-valued case.  In Section \ref{section:complex-case}, we discuss the modifications in the complex-valued case.  In Section \ref{section-numerical-results}, we present computational methods for the optimizationless problem with comparisons to PhaseLift.  We also simulate them to establish stability empirically.

\subsection{Notation}
We use boldface for variables representing vectors or matrices.  We use normal typeface for scalar quantities.  Let $z_{i,k}$ denote the $k$th entry of the vector $\bfz_i$.  For two matrices, let $\< \X, \Y \> = \tr(\Y^* \X)$ be the Hilbert-Schmidt inner product.  Let $\sigma_i$ be the singular values of the matrix $\X$.  We define the norms $$\|\X\|_p = \left(\sum_i \sigma_i^p\right)^{1/p}.$$  In particular, we write the Frobenius norm of $\X$ as $\|\X\|_2$.  We write the spectral norm of $\X$ as $\|\X\|$.

An $n$-vector $\bfx$ generates a decomposition of $\R^n$ or $\C^n$ into two subspaces.  These subspaces are the span of $\bfx$ and the span of all vectors orthogonal to $\bfx$.  Abusing notation, we write these subspaces as  $\bfx$ and $\bfx^\perp$. The space of $n$-by-$n$ matrices is correspondingly partitioned into the four subspaces $\bfx \, \otimes \, \bfx, \; \bfx \,\otimes \, \bfx^\perp, \; \bfx^\perp \, \otimes \, \bfx$, and $\bfx^\perp \, \otimes \, \bfx^\perp$, where $\otimes$ denotes the outer product. We write $T_\bfx$ for the set of symmetric matrices which lie in the direct sum of the first three subspaces, namely $T_\bfx = \{ \bfx \bfy^* + \bfy \bfx^* \mid  \bfy \in \R^n \text{ or } \C^n\}$. Correspondingly, we write $T^\perp_\bfx$ for the set of symmetric matrices in the fourth subspace.  We note that $T^\perp_\bfx$ is the orthogonal complement of $T_\bfx$ with respect to the Hilbert-Schmidt inner product. Let $\bfe_1$ be the first coordinate vector. For short, let $T = T_{\bfe_1}$ and $T^\bot = T^\bot_{\bfe_1}$.  We denote the projection of $\X$ onto $T$ as either $\mathcal{P}_T \X$ or $\XT$.  We denote projections onto $T^\perp$ similarly.  

We let $\I$ be the $n\times n$ identity matrix.  We denote the range of $\A^*$ by $\mathcal{R}( \A^*)$.

\section{Proof of Main Result}
\label{section-main-result}
Because of scaling and the property that the measurement vectors $\bfz_i$ come from a rotationally invariant distribution, we take $\bfx_0 = \bfe_1$ without loss of generality.  Because all measurements scale with the length $\|\bfz_i\|_2$, it is equivalent to establish the result for independent unit normal sensing vectors $\bfz_i$.  To prove Theorem \ref{theorem-exact-lifted}, we use an argument based on inexact dual certificates and $\ell_1$-isometry properties of $\A$.  This argument parallels that of \cite{CSV2011}.  We directly use the $\ell_1$-isometry properties they establish, but we require different properties on the inexact dual certificate.

\subsection{About Dual Certificates}

As motivation for the introduction of an inexact dual certificate in the next section, observe that if $\A$ is injective on $T$, and if there exists a (exact) dual certificate $\Y \in \mathcal{R}(\A^*)$ such that $$\YT = 0 \quad \text{ and } \quad \YTp \matrixgt 0,$$ then $\X_0$ is the only solution to $\A(\X) = \bfb$.  This is because $$
0 = \< \X - \X_0, \Y \> = \<\XTp, \YTp \> \Rightarrow \XTp = 0 \Rightarrow \X = \X_0,
$$
where the first equality is because $\Y \in \mathcal{R}(\A^*)$ and $\A(\X) = \A(\X_0)$.  The last implication follows from injectivity on $T$.  

Conceptually, $\Y$ arises as a Lagrange multiplier, dual to the constraint $\X \succeq 0$ in the feasibility problem
\[
\min \; 0  \quad \text{ such that } \quad  \A(\X) = \bfb, \qquad \X \succeq 0.
\]
Dual feasibility requires $\Y \succeq 0$. As visualized in Figure \ref{fig:dual-certificate-interpretation}a,  $\Y$ acts as a vector normal to a codimension-1 hyperplane that separates the lower-dimensional space of solutions $\{\A(\X) = b\}$ from the positive matrices not in $T$. The condition $\YTp \matrixgt 0$ is further needed to ensure that this hyperplane only intersects the cone along $T$,  ensuring uniqueness of the solution.

The nullspace condition $\YT = 0$ is what makes the certificate exact. As $\Y \in \mathcal{R}(\A^*)$, $\Y$ must be of the form $\sum_i \lambda_i \zizi$. The strict requirement that $\YT = 0$ would force the $\lambda_i$ to be complicated (at best algebraic) functions of all the $\bfz_j$, $j = 1,\ldots, m$. We follow \cite{CSV2011} in constructing instead an inexact dual certificate, such that $\YT$ is close to but not equal to $0$, and for which the $\lambda_i$ are more tractable (quadratic) polynomials in the $\bfz_i$. A careful inspection of the injectivity properties of $\A$, in the form of the RIP-like condition in \cite{CSV2011}, is what allows the relaxation of the nullspace condition on $\Y$.


\begin{figure}
	\begin{center}
  \begin{psfrags}
    \psfragscanon
    \psfrag{a}[rc]{$-\Ybar$}
    \psfrag{b}[rc]{$-\Y$}
    \psfrag{c}[cb]{}
    \psfrag{d}[lb]{}
    \psfrag{e}[lb]{}
    \psfrag{f}[cm]{Exact dual certificate}
    \psfrag{g}[cm]{Inexact dual certificate}
    \psfrag{h}[lb]{$\X_0$}
    \includegraphics[width=.75 \textwidth]{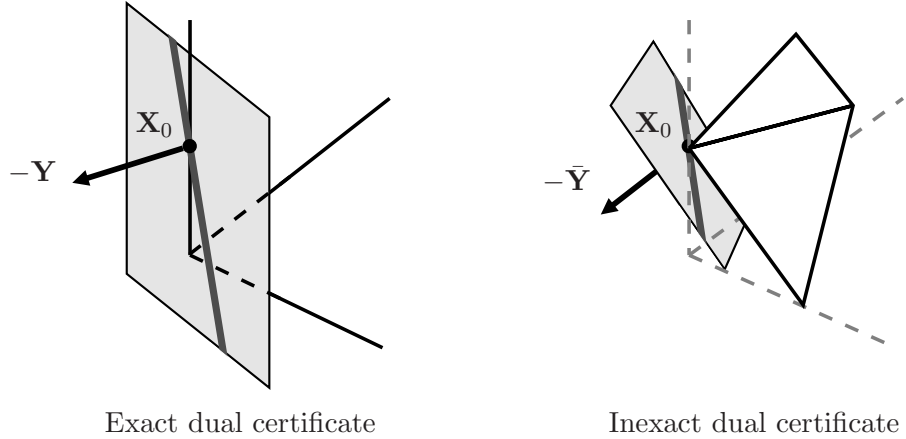}
  \end{psfrags}
  	\end{center}
  \caption{The graphical interpretation of the exact and inexact dual certificates.  The positive axes  represent the cone of positive matrices.  The thick gray line represents the solutions to $\A(\X) = \bfb$.  The exact dual certificate $\Y$ is a normal vector to a hyperplane that separates the space of solutions from positive matrices.   When the dual certificate is inexact, we use the fact
  that $\ell_1$-isometry properties imply $\X$ is restricted to the cone \eqref{cone-condition}.  The inexact dual certificate $\Ybar$ is normal to a hyperplane that separates $\X_0$ from the rest of this restricted cone.  As shown, the hyperplane normal to $\Ybar$ does not separate $\X_0$ from positive matrices.  }
  \label{fig:dual-certificate-interpretation}
\end{figure}

\subsection{Central Lemma on Inexact Dual Certificates}

With further information about feasible $\X$, we can relax the property that $\YT$ is exactly zero.  In \cite{CSV2011}, the authors show that all feasible $\X$ lie in a cone that is approximately  $\{ \|\XTp\|_1 \geq \|\XT - \X_0\|\}$, provided there are  $O(n)$ measurements. As visualized in Figure \ref{fig:dual-certificate-interpretation}b, $\Ybar$ acts as a vector normal to a hyperplane that separates $\X_0$ from the rest of this cone.    The proof of Theorem \ref{theorem-exact-lifted} hinges on the existence of such an inexact dual certificate, along with $\ell_1$-isometry properties that establish $\X$ is in this cone with high probability.
\begin{lemma}
\label{lemma-exact-dual-certificate}
Suppose that $\A$ satisfies 
\begin{alignat}{2}
m^{-1} \| \A(\X) \|_1 &\leq (1+ \delta) \|\X\|_1 \quad &&\text{ for all } \X \matrixgeq 0 \label{l1-isometry-psd},\\
m^{-1} \| \A(\X) \|_1 &\geq 0.94(1 - \delta) \|\X\| \quad &&\text{ for all } \X \in T, \label{l1-isometry-t} 
\end{alignat}
for some $\delta \leq 1/9$.  Suppose that there exists $\Ybar \in \mathcal{R} (\A^*)$ satisfying
\begin{align}
\|\YbarT\|_1 \leq 1/2 \quad \text{ and } \quad \YbarTp \matrixgeq \ITp. \label{inexact-dual-conditions}
\end{align}
Then, $\X_0$ is the unique solution to \eqref{exact-axb-lifted}.
\end{lemma}

\begin{proof}[Proof of Lemma \ref{lemma-exact-dual-certificate}]

Let $\X$ solve \eqref{exact-axb-lifted}, and let $\bfH = \X - \X_0$.
We start by showing, as in \cite{CSV2011}, that the $\ell_1$-isometry conditions \eqref{l1-isometry-psd}--\eqref{l1-isometry-t} guarantee solutions lie on the cone 
\begin{align}
\| \HTp  \|_1 \geq \frac{0.94 (1-\delta)}{1+\delta} \| \HT\|. \label{cone-condition}
\end{align}
This is because $$0.94 (1-\delta) \| \HT \| \leq m^{-1} \| \A (\HT) \|_1 = m^{-1} \| \A (\HTp) \|_1 \leq (1+\delta) \| \HTp \|_1,$$
where the equality comes from $0 = \A(\bfH) = \A(\HT) + \A(\HTp)$, and the two inequalities come from the $\ell_1$-isometry properties \eqref{l1-isometry-psd}--\eqref{l1-isometry-t} and the fact that $\HTp \matrixgeq 0$.

Because $\A(\bfH) = 0$ and $\Ybar \in \mathcal{R}(\A^*),$
\begin{align}
0 & = \< \bfH, \Ybar \> \nonumber \\
&= \< \HT, \YbarT \> + \< \HTp, \YbarTp \> \nonumber \\
&\geq \|\HTp \|_1 - \frac{1}{2}\|\HT\| \label{exact-bound-htp-ht}\\
&\geq \left( \frac{0.94 (1-\delta)}{1+\delta} - \frac{1}{2} \right) \|\HT\|, \label{exact-bound-ht}
\end{align}
where \eqref{exact-bound-htp-ht} and \eqref{exact-bound-ht} follow from  \eqref{inexact-dual-conditions} and \eqref{cone-condition}, respectively.  Because the constant in \eqref{exact-bound-ht} is positive, we conclude $\HT = 0$.  Then, \eqref{exact-bound-htp-ht} establishes $\HTp = 0$.
\end{proof}

\subsection{ Proof of Theorem \ref{theorem-exact-lifted} and Corollary \ref{corollary-exact}}
\label{section-proof-main-theorem}
We use  Lemma \ref{lemma-exact-dual-certificate} to prove Theorem \ref{theorem-exact-lifted} for real-valued signals.
\begin{proof}[Proof of Theorem \ref{theorem-exact-lifted}]
We need to show that \eqref{l1-isometry-psd}--\eqref{inexact-dual-conditions} hold with high probability if $m > c n \log n$ for some $c$.  Lemmas 3.1 and 3.2 in \cite{CSV2011} show that \eqref{l1-isometry-psd} and \eqref{l1-isometry-t} both hold with probability of at least $1 - 3 e^{-\gamma_1 m}$ provided $m > c_1 n$ for some $c_1$.  In section \ref{section-inexact-dual-certificate}, we construct $\Ybar \in \mathcal{R}({\A^*})$.  As per Lemma \ref{lemma-dual-certificate-t}, $\|\YbarT\|_1 \leq 1/2$ with probability at least $1 - e^{-\gamma_2 m / n}$ if $m > c_2 n$.  As per Lemma \ref{lemma-dual-certificate-tperp}, $\|\YbarTp - 2 \ITp \| \leq 1$ with probability at least $1 - 2 e^{-\gamma_2 m / \log n}$ if $m > c_3 n \log n$.  Hence, $\YbarTp \matrixgeq \ITp$ with at least the same probability.  Hence,  all of the conditions of Lemma \ref{lemma-exact-dual-certificate} hold with probability at least $1 - e^{-\gamma m/n}$ if $m > c n \log n$ for some $c$ and $\gamma$.
\end{proof}
\noindent The proof of Corollary \ref{corollary-exact} is immediate because, with high probability, Theorem \ref{theorem-exact-lifted} implies $$
A(\bfx_1) = A(\bfx_0) \Rightarrow \bfx_1\bfx_1^* = \bfx_0\bfx_0^* \Rightarrow \bfx_1 = e^{i \phi} \bfx_0.
$$

\section{Existence of Inexact Dual Certificate}
\label{section-inexact-dual-certificate}
To use Lemma \ref{lemma-exact-dual-certificate} in the proof of Theorem \ref{theorem-exact-lifted}, we need to show that there exists an inexact dual certificate satisfying \eqref{inexact-dual-conditions} with high probability.  Our inexact dual certificate vector is different from that in \cite{CSV2011}, but we use identical tools for its construction and analysis.  We also adopt similar notation.

We note that $\A^* \A ( \X) =  \sum_i \< \X, \bfz_i \bfz_i^* \> \bfz_i \bfz_i^*$, which can alternatively be written as 
$$
\A^* \A = \sum_{i=1}^m \bfz_i \bfz_i^* \otimes \bfz_i \bfz_i^*.
$$
We let $\Scal= \E[\bfz_i \bfz_i^* \otimes \bfz_i \bfz_i^*]$.  The operator $\Scal$ is invertible.  It and its inverse are given by
\begin{align}
\Scal(\X) &= 2 \X + \tr(\X) \I, \nonumber \\
\Sinv(\X) &= \frac{1}{2} \left( \X - \frac{1}{n+2} \tr(\X) \I \right). \label{s-inv-definition}
\end{align}
We define the inexact dual certificate 
\begin{align}
\Ybar = \frac{1}{m} \sum_{i=1}^m 1_{E_i} \Y_i, \label{defn-ybar}
\end{align}
where
\begin{align}
\Y_i &=\left [ \frac{3}{n+2} \|\bfz_i\|_2^2 - z_{i,1}^2 \right ] \bfz_i \bfz_i^*, \label{def_yi}\\
E_i &= \{ |z_{i,1}| \leq \sqrt{2 \beta \log n}\} \cap \{ \| \bfz_i \|_2 \leq \sqrt{3 n} \}.
\end{align}
Alternatively, we can write the inexact dual certificate vector as
\begin{align}
\Ybar = \frac{1}{m} \A^* \left( \oneE \circ \A \Sinv 2 (\I - \ee) \right),    \label{defn-ybar-second}
\end{align}
where $(\oneE)_i = 1_{E_i}$ and $\circ$ is the elementwise product of vectors.  In our notation, truncated quantities have overbars.  We subsequently omit the subscript $i$ in $\bfz_i$ when it is implied by context. 

\subsection{Motivation for the Dual Certificate}

For ease of understanding, we first consider a candidate dual certificate given by
\begin{align}
\widetilde{\Y} = \frac{1}{m} \A^* \A \Sinv 2 (\I - \ee ). \nonumber
\end{align}
The motivation for this candidate is twofold:  $\widetilde{\Y} \in \mathcal{R}(\A^*)$,  and  $\widetilde{\Y} \approx 2(\I - \ee)$ as $m \to \infty$ because $\E[ \A^* \A] = m \mathcal{S}$.  In this limit, $\widetilde{Y}$ becomes an exact dual certificate.  For finite $m$, it should be close but inexact.  We can write $$\widetilde{\Y} = \frac{1}{m}  \sum_i \Y_i,$$ where  $\Y_i$ is an independent sample of the random matrix
\begin{align*}
\left [ \frac{3}{n+2} \|\bfz\|_2^2 - z_1^2 \right ] \bfz \bfz^*,
\end{align*}
where $\bfz \sim \mathcal{N}(0, \I)$.  Because the vector Bernstein inequality requires bounded vectors, we truncate the dual certificate in the same manner as \cite{CSV2011}.  That is, we consider $1_{E_i} \Y_i$, completing the derivation of \eqref{defn-ybar}.

\subsection{Bounds on $\Ybar$}

We define $\pi(\beta) = \PP(E^c)$, where $E$ is the event given by 
\begin{align}E &= \{ |z_{1}| \leq \sqrt{2 \beta \log n}\} \cap \{ \| \bfz \|_2 \leq \sqrt{3 n} \},
\end{align}
where $\bfz \sim \mathcal{N}(0, \I)$.
In \cite{CSV2011}, the authors provide the bound $\pi(\beta) \leq \PP( |z_1| > \sqrt{2 \beta \log n}) + \PP(\|\bfz\|_2^2 > 3n) \leq n^{-\beta} + e^{-n/3}$, which holds if $2 \beta \log n \geq 1$.

We now present two lemmas that establish that $\Ybar$ is approximately $2(\I - \ee)$, and is thus  an inexact dual certificate satisfying  \eqref{inexact-dual-conditions}.

\begin{lemma}
\label{lemma-dual-certificate-t}
Let $\Ybar$ be given by \eqref{defn-ybar}.  There exists positive $\gamma$ and $c$ such that for sufficiently large $n$
$$
\PP \left( \left \| \YbarT  \right \|_1  \geq \frac{1}{2}\right) \leq \exp \left(-\gamma \frac{m}{n} \right)
$$
if $m \geq c n$.
\end{lemma}

\begin{lemma}
\label{lemma-dual-certificate-tperp}
Let $\Ybar$ be given by \eqref{defn-ybar}.  There exists positive  $\gamma$ and $c$ such that for sufficiently large $n$
$$
\PP \left( \left \| \YbarTp - 2 \ITp  \right \|  \geq1 \right) \leq 2 \exp \left(-\gamma \frac{m}{\log n} \right)
$$
if $m \geq c n \log n$.

\end{lemma}

\subsection{Proof of Lemma \ref{lemma-dual-certificate-t}:  $\Ybar$ on $T$}

We prove Lemma \ref{lemma-dual-certificate-t} in a way that parallels the corresponding proof in \cite{CSV2011}.  Observe that 
 $$\| \YbarT\|_1 \leq \sqrt{2}  \| \YbarT  \|_2 \leq 2 \| \YbarT \bfe_1\|_2,$$
where the first inequality follows because $\YbarT$ has rank at most 2, and the second inequality follows because $\YbarT$ can be nonzero only in its first row and column. We can write 
$$
\YbarT \bfe_1 = \frac{1}{m} \sum_{i=1}^m \ybar_i,
$$
where $\ybar_i =  \bfy_i 1_{E_i}$,  and $\bfy_i$ are independent samples of
$$
\bfy = \left [ \frac{3}{n+2} \|\bfz\|_2^2 - z_1^2 \right ] z_1 \bfz =:  \xi z_1 \bfz.
$$
To bound the $\ell_2$ norm of $\YbarT \bfe_1$, we use the Vector Bernstein inequality on $\ybar_i$.

\begin{theorem}[Vector Bernstein inequality]
Let $\bfx_i$ be a sequence of independent random vectors and set  $V \geq \sum_i \E\|\bfx_i\|_2^2$.  Then for all $t \leq V / \max \| \bfx_i \|_2$, we have
$$
\PP\left (\left \| \sum_i (\bfx_i - \E \bfx_i )  \right \|_2 \geq \sqrt{V} + t \right ) \leq e^{-t^2/4V}.
$$
\end{theorem}
 
\noindent In order to apply this inequality, we need to compute $\max \| \ybar \|_2$, $\E \ybar$, and $ \E \| \ybar \|_2$, where $\ybar = \bfy 1_E$.

First, we compute $\max \| \ybar \|_2$.  On the event $E$, $|z_1| \leq \sqrt{2 \beta \log n}$ and $\|\bfz\|_2 \leq \sqrt{3 n}$.  If $n$ is large enough that $2 \beta \log n \geq 9$, then $| \xi | \leq 2 \beta \log n$. Thus, 
$$
\| \ybar \|_2 \leq \sqrt{24 n} (\beta \log n)^{3/2}
$$
for sufficiently large $n$.

Second, we find an upper bound for $\E \ybar$.  Note that $\E y_1 = 0$ because
\begin{align*}
\E[z_1^4] &=3,\\
\E [z_1^2 \|\bfz\|_2^2] &= n+2.
\end{align*}
By symmetry, every entry of $\ybar$ has zero mean except the first.  Hence, 
$$
\| \E \ybar \|_2 = | \E \bar{y}_1 | = | \E (y_1 - y_1 1_{E^c} ) | = | \E y_1 1_{E^c} | \leq \sqrt{\PP (E^c) } \sqrt{\E y_1^2} = \sqrt{\pi(\beta)} \sqrt{ \E y_1^2}.
$$
Computing,
$$
y_1^2 = (\xi z_1^2)^2 =  z_1^8 - \frac{6}{n+2} z_1^6 \|z\|_2^2 + \frac{9}{(n+2)^2} z_1^4 \|z\|_2^4,
$$
we find 
$$
\E y_1^2 \leq 44,
$$
where we have used
\begin{alignat}{2}
&\E[ z_1^8] &&= 105, \\
&\E[z_1^6 \| \bfz \|_2^2 ] &&= 15n + 90, \label{ez16z22}\\
&\E[z_1^4 \| \bfz \|_2^4 ] &&= 3n^2 + 30 n + 72. \label{ez14z24}
\end{alignat}
Thus, 
\begin{align}
\| \E \ybar \|_2 \leq \sqrt{44(n^{-\beta} + e^{-n/3})}. \label{exp-ybar-2norm}
\end{align}

Third, we find an upper bound for $\E \| \ybar \|_2^2$.  Because $\| \ybar \|_2^2 \leq \| \bfy \|_2^2$, we write out
\begin{align*}
\| \bfy \|_2^2 &= \xi^2 z_1^2  \| \bfz \|_2^2  =
 z_1^6 \| \bfz \|_2^2 - \frac{6}{n+2} z_1^4 \|\bfz \|_2^4 + \frac{9}{(n+2)^2} z_1^2 \|\bfz\|_2^6. 
\end{align*}
Hence, 
\begin{align}
\E [ \| \bfy \|_2^2 ] &= (15 n + 90) - \frac{6}{n+2} (3n^2 + 30 n + 72) + \frac{9}{(n+2)^2} (n+2)(n+4)(n+6) \\
&\leq 8n+16,
\end{align}
where we have used \eqref{ez16z22}, \eqref{ez14z24}, and
\begin{alignat}{2}
&\E[z_1^2 \| \bfz \|_2^6 ] &&= (n+2)(n+4)(n+6).
\end{alignat}
Applying the vector Bernstein inequality with $V=m(8n+16)$, we have that for all $t \leq (8n+16) / [ \sqrt{24n}(\beta \log n)^{3/2}]$, 
$$
\PP \left( \frac{1}{m} \left  \| \sum_i \ybar_i - \E \ybar_i \right \|_2 \geq \sqrt{\frac{8n+16}{m}} + t \right) \leq \exp \left(- \frac{mt^2}{4(8n+16)} \right ).
$$
Using the triangle inequality and \eqref{exp-ybar-2norm}, we get
$$
\PP \left( \frac{1}{m} \left  \| \sum_i \ybar_i \right \|_2 \geq \sqrt{44 (n^{-\beta} + e^{-n/3}} ) + \sqrt{\frac{8n+16}{m}} + t \right) \leq \exp \left(- \frac{mt^2}{4(8n+16)} \right ).
$$
Lemma \ref{lemma-dual-certificate-t} follows by choosing $t, \beta$, and $m \geq c n$ where $n$ and $c$ are large enough that 
$$
 \sqrt{44 (n^{-\beta} + e^{-n/3}} ) + \sqrt{\frac{8n+16}{m}} + t \leq \frac{1}{4}.
$$

\subsection{Proof of Lemma \ref{lemma-dual-certificate-tperp}: $\Ybar$ on $T^\perp$}

We prove Lemma \ref{lemma-dual-certificate-tperp} in a way that parallels the corresponding proof in \cite{CSV2011}.  We write
$$\YbarTp - 2 \ITp = \frac{1}{m} \sum_i  (\W_i 1_{E_i} -2 \ITp 1_{E_i^c}),
$$
where $\W_i$ are independent samples of 
\begin{align}
\W = \left [ \frac{3}{n+2} \|\bfz\|_2^2 - z_1^2 \right ] \PTp ( \zz) - 2 \ITp.
\end{align}
We decompose $\W$ into the three terms
\begin{align}
\W &\phantom{:}= -\left[z_1^2 -1  \right] \PTp(\zz) + 3 \left[  \frac{1}{n+2} \|\bfz\|_2^2 - 1 \right] \PTp(\zz) + 2 (\PTp{\zz} - \ITp)\\
&:= \Wzero + \Wone + \Wtwo.
\end{align}
Letting $\Wkbar_i = \Wk_i 1_{E_i}$, it suffices to show
that with high probability
\begin{align}
\frac{1}{m}\left \| \sum_i 2\ITp 1_{E_i^c}  \right \| \leq \frac{1}{4} \text{ and, }
\frac{1}{m}\left \|\sum_i \bar{\W}_i^{(k)}    \right \| \leq \frac{1}{4} \text{ for $k=0, 1, 2.$} 
\end{align}

\subsubsection{Bound on $\ITp 1_{\Eic}$}

We show that $m^{-1} \| \sum_i \ITp 1_{\Eic}\| =  m^{-1} \sum_i 1_{\Eic} $ is small with probability at least $1 - 2 e^{-\gamma m}$ for some constant $\gamma > 0$.  To do this, we use the scalar Bernstein inequality.

\begin{theorem}[Bernstein inequality]
Let $\{ X_i \}$ be a finite sequence of independent random variables.  Suppose that there exists $V$ and $c$ such that for all $X_i$ and all $k\geq 3$,
$$
\sum_i \E |X_i|^k \leq \frac{1}{2} k! V c_0^{k-2}.
$$
Then for all $t\geq 0$,
\begin{align}
\PP \left( \left | \sum_i X_i - \E X_i \right | \geq t \right) \leq 2 \exp \left( - \frac{t^2}{2V + 2 c_0 t} \right ). \label{bernstein-scalar}
\end{align}
\end{theorem}

\noindent Observing that $\E | 1_{\Eic}|^k = \E 1_{\Eic} = \pi(\beta)$, we apply the Bernstein inequality with $V = \pi(\beta) m$ and $c_0 = 1/3$.  Thus,
$$
\PP \left(  \left |\frac{1}{m} \sum_i  1_{\Eic} - \pi(\beta) \right |  \geq t \right ) \leq 2 \exp\left( - \frac{m t^2}{2 \pi(\beta) + 2 t /3}\right).
$$
Using the triangle inequality and taking  $t$ and $\beta$ such that $\pi(\beta) + t \leq 1/8$ for sufficiently large $n$, we get 
$$
\PP \left(  \left |\frac{1}{m} \sum_i  1_{\Eic}  \right |  \geq \frac{1}{8} \right ) \leq 2 \exp\left( - \gamma m \right)
$$
for a $\gamma > 0$.

\subsubsection{Bound on $\Wzerobar$}
We show $m^{-1} \|\sum_i \Xzerobar\|$ is small with probability at least $1 - 2 \exp( -\gamma/\log n )$.  We write this norm as a supremum over all unit vector perpendicular to $\bfe_1$:
\begin{align}
\left \|\sum_i \Wzerobar \right \| = \sup_{\bfu \perp \bfe_1, \|\bfu\|=1} \left | \sum_i \< \bfu, \Wzerobar_i \bfu \> \right|, \label{w0-norm-sup} 
\end{align}
To control the supremum, we follow the same reasoning as in \cite{CSV2011}.  We bound $ \sum_i \< \bfu, \Wzerobar_i \bfu \>$ for fixed $\bfu$ and apply a covering argument over the  sphere of $\bfu$'s.  We write
$$
 \sum_i \< \bfu, \Xzerobar_i \bfu \> = \sum_i \eta_i 1_{E_i},
$$
where $\eta_i$ are independent samples of 
$$
\eta = -\left[ z_1^2  -1  \right ] \<\bfz, \bfu \>^2.
$$
To apply the scalar Bernstein inequality, we compute  $\E | \eta 1_E |^k$.  Because $\bfu \perp \bfe_1$, $z_1$ and $\< \bfz, \bfu \>$ are independent.  Hence,
$$
\E | \eta 1_E|^k \leq \E | (z_1^2 - 1) 1_{E }|^k \E | \< \bfz, \bfu \> |^{2k}.
$$
Bounding the first factor, we get
$$
\E |(z_1^2 - 1) 1_E|^k  = \E |(z_1^2 - 1)^{k-2} 1_E  (z_1^2 - 1)^{2} | \leq (2 \beta \log n)^{k-2} \E(z_1^2 -1)^2 = 2 (2 \beta \log n)^{k-2}.
$$
Observing that $\< \bfz, \bfu\>$ is  a chi-squared variable with one degree of freedom, we have
$$
\E | \<\bfz, \bfu \> |^{2k} = 1 \times 3 \times \ldots \times (2k-1) \leq 2^k k!
$$
Applying the scalar Bernstein inequality with $V=16m$ and $c_0 = 4 \beta \log n$, we get
$$
\PP \left( \frac{1}{m} \left |\sum_i \eta_i 1_{E_i} - \E[ \eta_i 1_{E_i}]   \right | \geq t \right) \leq 2 \exp \left( - \frac{m t^2}{2(16+4 \beta t \log n)} \right ).
$$
Because $\E \eta_i = 0$, we get
$$
| \E \eta_i 1_{E_i} | = | \E \eta_i 1_{\Eic} | \leq \sqrt{\PP(\Eic)} \sqrt{ \E \eta_i^2} = 2\sqrt{ \pi(\beta)},
$$
where we have used $\E (1-z_1^2)^2 = 2$, and $\E|\< \bfz, \bfu \> |^4 = 3$.  Hence, 
$$
\PP \left( \frac{1}{m} \left |\sum_i \eta_i 1_{E_i}  \right | \geq t + 2\sqrt{ \pi(\beta)} \right) \leq 2 \exp \left( - \frac{m t^2}{2(16+4 \beta t \log n)} \right ).
$$
Taking $t, \beta, m\geq c_1 n$ with $n$ large enough so that $t + 2 \sqrt{\pi(\beta)} \leq 1/8$, we have
$$
\PP \left( \frac{1}{m} \left |\sum_i \eta_i 1_{E_i}  \right | \geq 1/8 \right) \leq 2 \exp \left(- \gamma' \frac{m}{\log n} \right ),
$$
for some $\gamma'>0$.  To complete the bound on \eqref{w0-norm-sup}, we use Lemma 4 in \cite{Ver2010}:
$$
\sup_\bfu \left | \<  \bfu, \Wzerobar \bfu \>   \right | \leq 2 \sup_{\bfu \in \mathcal{N}_{1/4}} \left |  \<  \bfu, \Wzerobar \bfu \>  \right |,
$$
where $\mathcal{N}_{1/4}$ is a 1/4-net of the unit sphere of vectors $\bfu \perp \bfe_1$.  As $| \mathcal{N}_{1/4}| \leq 9^n$, a union bound gives
$$
\PP \left( \frac{1}{m} \left |\sum_i \eta_i 1_{E_i}  \right | \geq 1/8 \right) \leq 9^n \cdot 2 \exp \left(- \gamma' \frac{m}{\log n} \right ).$$
Hence, 
$$
\PP \left( \frac{1}{m} \left \|\sum_i \Wzerobar \right \| \geq \frac{1}{4} \right) \leq 2 \exp \left(- \gamma m /\log n \right ) 
$$
for some $\gamma>0$.
\subsubsection{Bounds on $\Wonebar$ and $\Wtwobar$}
The bound for the $\|\sum_i \Wonebar\|$ term is similar.  We  write
$$
 \sum_i \< \bfu, \Wonebar_i \bfu \> = \sum_i \eta_i 1_{E_i},
$$
where $\eta_i$ are independent samples of 
$$
\eta =  3 \left[ \frac{\|\bfz\|_2^2}{n+2} -1  \right ] \<\bfz, \bfu \>^2.
$$
We can bound $\E | \eta_i 1_E |^k \leq 12^k k!$ because $\|\bfz \|_2^2 \leq 3n$ on $E$.  Applying the scalar Bernstein inequality with $c_0 = 12$ and $V = 288 m$ gives
$$
\PP \left( \frac{1}{m} \left |\sum_i \eta_i 1_{E_i} - \E[ \eta_i 1_{E_i}]   \right | \geq t \right) \leq 2 \exp \left( - \frac{m t^2}{2(288+ 12 t)} \right ).
$$
The rest of the bound is similar to that of $\|\sum_i \Xzerobar\|$  above.

Finally, we also bound  $\|\sum_i \Wtwobar\|$ similarly.  We  write
$$
 \sum_i \< \bfu, \Wtwobar_i \bfu \> = \sum_i \eta_i 1_{E_i},
$$
where $\eta_i$ are independent samples of 
$$
\eta = 2\<\bfz, \bfu \>^2 - 2.
$$
Observing that 
$$
\E |\eta_i 1_E|^k \leq 4^k k!,
$$
we apply the scalar Bernstein inequality with $c_0 = 4$ and $V = 32 m$, giving
$$
\PP \left( \frac{1}{m} \left |\sum_i \eta_i 1_{E_i} - \E[ \eta_i 1_{E_i}]   \right | \geq t \right) \leq 2 \exp \left( - \frac{m t^2}{2(32+ 4 t)} \right ).
$$
The rest of the bound is as above.

\section{Stability}
\label{section-stability}
We now prove Theorem \ref{theorem-noisy}, establishing the stability of the matrix recovery problem \eqref{noisy-axb-lifted}.  We also prove Corollary \ref{corollary-noisy}, establishing the stability of the vector recovery problem \eqref{noisy-axb}.
As in the exact case, the proof of Theorem \ref{theorem-noisy} hinges on the $\ell_1$-isometry properties \eqref{l1-isometry-psd}--\eqref{l1-isometry-t} and the existence of an inexact dual certificate satisfying \eqref{inexact-dual-conditions}.   For stability, we use the  additional property that $\Y = \A^* \lambda$ for a  $\lambda$ controlled in $\ell_2$.
It suffices to establish an analogue of Lemma \ref{lemma-exact-dual-certificate} along with a bound on $\|\lambda\|_2$.

\begin{lemma}
\label{lemma-noisy-dual-certificate}
Suppose that $\A$ satisfies   \eqref{l1-isometry-psd} -- \eqref{l1-isometry-t} and  there exists $\Y = \A^* \lambda$ satisfying \eqref{inexact-dual-conditions} and $\|\lambda\|_1 \leq 5$.
Then, 
$$\X \matrixgeq 0 \text{ and } \| \A(\X) - \bfb \|_2 \leq \eps  \|\X_0\|_2 \Longrightarrow 
\| \X - \X_0\|_2 \leq C \eps  \|\X_0\|_2,
$$
for some $C>0$.
\end{lemma}

\begin{proof}[Proof of Lemma \ref{lemma-noisy-dual-certificate}]
As before, we take $\bfx_0 = \bfe_1$ and $\X_0 = \ee$ without loss of generality.    
Consider any $\X \matrixgeq 0$ such that $\|\A(\X) - \bfb\|_2\leq \eps $, and let $\bfH = \X - \X_0$.  Whereas $\A(\bfH) = 0$ in the noiseless case, it is now of order $\eps$ because
\begin{align}
\|\A (\bfH) \|_2 \leq \| \A(\X - \bfb) \|_2 + \| \A(\X_0 - \bfb)\|_2 \leq 2 \eps. \label{AH-small}
\end{align}

\noindent Similarly, $| \<\bfH, \Y \>|$ is  also of order $\eps$ because
\begin{align}
| \<\bfH, \Y \>| &= | \< \A(\bfH), \lambda \> \| \leq \| \A (\bfH) \|_\infty \ \|\lambda \|_1 \leq \| \A (\bfH) \|_2 \ \|\lambda \|_1 \nonumber \leq 10 \eps.
\end{align}
Analogous to the proof of Lemma \ref{lemma-exact-dual-certificate}, we use \eqref{inexact-dual-conditions} to compute that
\begin{align}
10 \eps &\geq \< \bfH, \Y \> \geq \| \HTp \|_1 - \frac{1}{2} \|\HT\|. \label{upper-bound-HY}
\end{align}
Using the $\ell_1$-isometry properties \eqref{l1-isometry-psd} -- \eqref{l1-isometry-t}, we have
\begin{align}
0.94(1-\delta) \| \HT \| \leq m^{-1} \| \A (\HT) \|_1 &\leq m^{-1} \| \A (\bfH) \|_1 + m^{-1} \| \A (\HTp) \|_1 \nonumber \\
&\leq m^{-1/2} \| \A (\bfH) \|_2 + (1+ \delta) \| \HTp \|_1 \nonumber \\
&\leq 2 \eps m^{-1/2} + (1+\delta) \| \HTp \|_1 \label{lowerbound-HT}.
\end{align}
Thus \eqref{upper-bound-HY} becomes
\begin{align}
\left( 10 + \frac{m^{-1/2}}{0.94(1-\delta)}  \right) \eps &\geq \left( 1 - \frac{1+\delta}{2 \cdot 0.94 ( 1-\delta)} \right) \| \HTp \|_1,
\end{align}
which, along with \eqref{lowerbound-HT},  implies 
\begin{align}
\|\HTp \|_1 &\leq C_0 \eps \text{ and } \|\HT \| \leq  C_1 \eps
\end{align}
for some $C_0,C_1 > 0$.  Recalling that $\HT$ has rank at most 2,
$$
\|\bfH \|_2 \leq \|\HT\|_2 + \| \HTp \|_2 \leq \sqrt{2} \|\HT\| + \| \HTp\|_1 \leq(\sqrt{2} C_1 + C_0 ) \eps \leq C \eps.
$$

\end{proof}

\subsection{Dual Certificate Property}
It remains to show $\| \lambda \|_1 \leq 5$ for $\Ybar = \A^* \lambda$.  From \eqref{defn-ybar-second}, we identify $\lambda  = m^{-1} (\mathbf{1}_{E} \circ \A \Sinv 2(\I-\ee))$. 
 Computing, 
\begin{align}
\|\lambda\|_1 &= m^{-1}  \| \oneE \circ  \A \Sinv 2(\I - \ee) \|_1  \nonumber \\
&\leq m^{-1}  \| \A \Sinv 2(\I - \ee) \|_1  \nonumber \\
&\leq m^{-1}  \left \|\A \left ( \frac{3}{n+2} \I \right) - \A \left( \ee \right)  \right \|_1 \label{HY-substitute-Sinv}\\
&\leq (1
+\delta)  \left(\left \| \frac{3}{n+2} \I \right\|_1 + \left \|  \ee \right \|_1 \right) \label{HY-use-isometry}\\
&\leq 4 (1+\delta),   \nonumber
\end{align}
where \eqref{HY-substitute-Sinv} follows from \eqref{s-inv-definition}, and \eqref{HY-use-isometry} follows from  the triangle inequality and the $\ell_1$-isometry property \eqref{l1-isometry-psd}.  Hence $\|\lambda\|_1 \leq 5$.
\subsection{Proof of Corollary \ref{corollary-noisy}}

Now we  prove Corollary \ref{corollary-noisy},  showing that stability of the lifted problem \eqref{noisy-axb-lifted} implies stability of the unlifted problem \eqref{noisy-axb}.  As before, we take $\bfx_0= \bfe_1$ without loss of generality.  Hence $\|\X_0\|_2 = 1$.  
Lemma \ref{lemma-noisy-dual-certificate} establishes that $\|\X - \X_0 \| \leq C_0 \eps$. 
Recall that $\X_0 =\xnotxnotstar$.   Decompose $X = \sum_j \lambda_j \bfv_j \bfv_j^t$ with unit-normalized eigenvectors $\bfv_j$ sorted by decreasing eigenvalue.  By Weyl's perturbation theorem,
\begin{align}
\max \left \{ |1-\lambda_1 |, |\lambda_2|, \ldots, |\lambda_n| \right \} \leq C_0 \eps. \label{weyl}
\end{align}
Writing 
\begin{align}
\X_0 - \vv = (\X_0 - \X) + \left( (\lambda_1 -1) \vv + \sum_{j=2}^m \lambda_j \bfv_j \bfv_j^* \right),
\end{align}
we use the triangle inequality to form the spectral bound
$$
\| \X_0 - \vv \| \leq 2 C_0 \eps.
$$
Noting that
$$
1 -| \< \bfx_0,\bfv \>|^2 = \frac{1}{2}\| \X_0 - \vv \|_2^2 \leq \| \X_0 - \vv \|^2 \leq 4 C_0^2 \eps^2,
$$
we conclude 
$$
\| \bfx_0 - \bfv \|_2^2 = 2 - 2 \< \bfx_0, \bfv \> \leq 8 C_0^2 \eps^2.
$$

\section{Complex Case}
\label{section:complex-case}
The proof of Theorems \ref{theorem-exact-lifted} and \ref{theorem-noisy} are analogous to the complex-valued cases.  There are a few minor differences, as outlined and proved in \cite{CSV2011}.  The sensing vectors are assumed to be of the form $\Re \bfz_i \sim \mathcal{N}(0, \I)$ and $\Im \bfz_i \sim \mathcal{N}(0, \I)$.
The $\ell_1$-isometry conditions for complex $\A$ have weaker constants.  Lemma \ref{lemma-exact-dual-certificate} becomes
\begin{lemma}
\label{lemma-exact-dual-certificate-complex}
Suppose that $\A$ satisfies 
\begin{alignat}{2}
m^{-1} \| \A(\X) \|_1 &\leq (1+ \delta) \|\X\|_1 \quad &&\text{ for all } \X \matrixgeq 0, \nonumber\\
m^{-1} \| \A(\X) \|_1 &\geq 0.828(1 - \delta) \|\X\| \quad &&\text{ for all } \X \in T, \nonumber 
\end{alignat}
for some $\delta \leq 3/13$.  Suppose that there exists $\Ybar \in \mathcal{R} (\A^*)$ satisfying
\begin{align}
\|\YbarT\|_1 \leq 1/2 \quad \text{ and } \quad \YbarTp \matrixgeq \ITp. \nonumber
\end{align}
Then, $\X_0$ is the unique solution to \eqref{exact-axb-lifted}.
\end{lemma}
\noindent The proof of this lemma is identical to the real-valued case.  The conditions of the lemma are satisfied with high probability, as before.  

The construction of the inexact dual certificate is slightly different because $\mathcal{S}(\X) = \X + \tr(\X) \I$ and $\Sinv(\X) = \X - \frac{1}{n+1} \tr(\X) \I$.  As a result
$$
\Y_i =\left [ \frac{4}{n+1} \|\bfz_i\|_2^2 - 2 |z_{i,1}|^2 \right ] \bfz_i \bfz_i^*.\\
$$
The remaining modifications are identical to those in \cite{CSV2011}, and we refer interested readers there for details.  

\section{Numerical  Simulations}
\label{section-numerical-results}

In this section, we show that the optimizationless perspective allows for additional numerical algorithms that are unavailable for PhaseLift directly.  These methods give rise to simpler algorithms with less or no parameter tuning.  We demonstrate successful recovery under Douglas-Rachford and Nesterov algorithms, and we empirically show that the convergence of these algorithms is linear.

\subsection{Optimization Framework}

From the perspective of nonsmooth optimization, PhaseLift and the optimizationless feasibility problem can be viewed as a two-term minimization problem
\begin{align}
\min_\X F(\X) + G(\X).
\end{align}
See, for example, the introduction to \cite{RFP2011}.
Numerical methods based on this splitting include Forward-Backward, ISTA, FISTA, and Douglas-Rachford \cite{RFP2011, CP2010, BT2009, DR1956}.  If $F$ is smooth, it enables a forward step based on a gradient descent.  Nonsmooth terms admit backward steps involving proximal operators. We recall that the proximal operator for a function $G$ is given by
\begin{align}
\text{prox}_G(\X) &= \text{argmin}_\Y \frac{1}{2} \|\X - \Y\|^2 + G(\Y),
\end{align}
and we note that the proximal operator for a convex indicator function is the projector onto the indicated set.

PhaseLift can be put in this two-term form by softly enforcing the data fit.  That gives the minimization problem
\begin{align}
\min_\X \underbrace{\frac{1}{2} \| \A(\X) - \bfb \|^2 + \lambda \ \text{tr}(\X)}_F + \underbrace{\indpos(\X)}_G \label{splitting-phaselift-soft}
\end{align}
 where $\indpos$ is the indicator function that is zero on the positive semidefinite cone and infinite otherwise, and where $\lambda$ is small and positive.  If $\lambda = 0$, \eqref{splitting-phaselift-soft} reduces to the optimizationless feasibility problem.  The smoothness of  $F$ enables methods that are forward on $F$ and backward on $G$.  As a representative of this class of methods, we will consider a Nesterov iteration for our simulations below.
 
 The optimizationless view suggests the splitting
\begin{align}
\min_\X \underbrace{\indaxb(\X)}_F + \underbrace{\indpos(\X)}_G. \label{splitting-optimizationless-hard}
\end{align}
where the data fit term is enforced in a hard manner by the indicator function $\indaxb$.  Because of the lack of smoothness, we can only use the proximal operators for $F$ and $G$.  These operators are projectors on to the affine space $\A(\X) = \bfb$ and $X\matrixgeq 0$, which we denote by 
 $\Paxb$ and $\Ppsd$, respectively.  

The simplest method for \eqref{splitting-optimizationless-hard} is Projection onto Convex Sets (POCS), which is given by the backward-backward iteration $\X_{n+1} = \Ppsd \Paxb \X_n.$  Douglas-Rachford iteration often gives superior performance than POCS, so we consider it as a representative of this class of backward-backward methods.

A strength of the optimizationless perspective is that it does not require as much parameter tuning as PhaseLift.  For example, formulation \eqref{splitting-phaselift-soft} requires a numerical choice for $\lambda$.  Nonzero $\lambda$ will generally change the minimizer.  It is possible to consider a sequence of problems with varying $\lambda$, or perhaps to create a schedule of $\lambda$ within a problem, but these considerations are unnecessary because the optimizationless perspective says we can take $\lambda = 0$.  
In particular,  formulation \eqref{splitting-optimizationless-hard} has the further strength of requiring no parameters at all.

We note that PhaseLift could alternatively give rise to the two-term splitting
\begin{align}
\min \underbrace{\text{tr}(\X)}_F + \underbrace{\indpos(\X) + \indaxb(\X)}_G, \label{splitting-phaselift-hard}
\end{align}
where the data fit term is enforced in a hard manner.  An iterative approach with this splitting would have an inner loop which approximates the proximal operator of $G$.  This inner iteration is equivalent to solving the optimizationless problem.

\subsection{Numerical Results}

First, we present a Douglas-Rachford \cite{CP2010} approach for finding $\X \in \{\X \matrixgeq 0\} \cap \{ \A (\X) \approx \bfb\}$ by the splitting \eqref{splitting-optimizationless-hard}.  It is given by the iteration
\begin{align}
\X_0 &= \Y_0 = 0 \label{dr-first}\\
\Y_n &= \Paxb(2 \X_{n-1} - \Y_{n-1}) - \X_{n-1} + \Y_{n-1}\\
\X_n &= \Ppsd(\Y_n) \label{dr-last}
\end{align}
where $\Ppsd$ is the projector onto the positive semi-definite cone of matrices, and $ \Paxb $ is the projector onto the affine space of solutions to $\A (\X) = \bfb$. In the classically underdetermined case, $m < \frac{(n+1)n}{2}$, we can write 
$$\Paxb \X =  \X - \A^* (\A \A^*)^{-1} \A (\X) + \A^*(\A \A^*)^{-1} \bfb.$$
In the case that $m \geq \frac{(n+1)n}{2}$, we interpret $\Paxb$ as the least squares solution to $\A(\X) = b$.

Second, we present a Nesterov gradient-based method for solving the problem \eqref{splitting-phaselift-soft}.
Letting ${g(\X) = \frac{1}{2} \| \A(\X) - \bfb\|^2 + \lambda \ \text{tr} (\X)}$, we consider the following Nesterov iteration \cite{CESV2011} with constant step size $\alpha$:
\begin{align}
\X_0 &= \Y_0 = 0 \label{nesterov-first}\\
\X_n &= \Ppsd(\Y_{n-1} - \alpha \nabla g(\Y_n-1))\\
\theta_n &= 2 \left( 1 + \sqrt{1 + 4 / \theta^2_{n-1}} \right)^{-1}\\
\beta_n &= \theta_n (\theta^{-1}_{n-1} - 1)\\
\Y_n &= \X_n + \beta_n(\X_n - \X_{n-1}) \label{nesterov-last}
\end{align}

For our simulations, we consider $\bfx_0 \in \R^n$ sampled uniformly at random from the unit sphere.  We take independent, real-valued $\bfz_i \sim \mathcal{N}(0,\I)$, and let the measurements $\bfb$ be subject to additive Gaussian noise corresponding to $\eps = 1/10$.   We let $n$ vary from $5$ to $50$ and let $m$ vary from $10$ to $250$.  We define the recovery error as $\|\X - \X_0\|_2 / \|\X_0\|_2$.  

Figure \ref{fig:phase-transition} shows the average recovery error for the optimizationless problem under the Douglas-Rachford method and the Nesterov method over a range of values of $n$ and $m$.  For the Nesterov method, we consider the optimizationless case of $\lambda = 0$, and we let the step size parameter $\alpha = 2\cdot 10^{-4}$.
 Each pair of values was independently sampled 10 times, and both methods were run for 1000 iterations.   The plot shows that the number of measurements needed for recovery is approximately linear in $n$, significantly lower than the amount for which there are an equal number of measurements as unknowns.  The artifacts around the curve $m = \frac{n(n+1)}{2}$ appear because the problem is critically determined, and the only solution to the noisy $\A(\X) = \bfb$ is not necessarily positive in that case.  

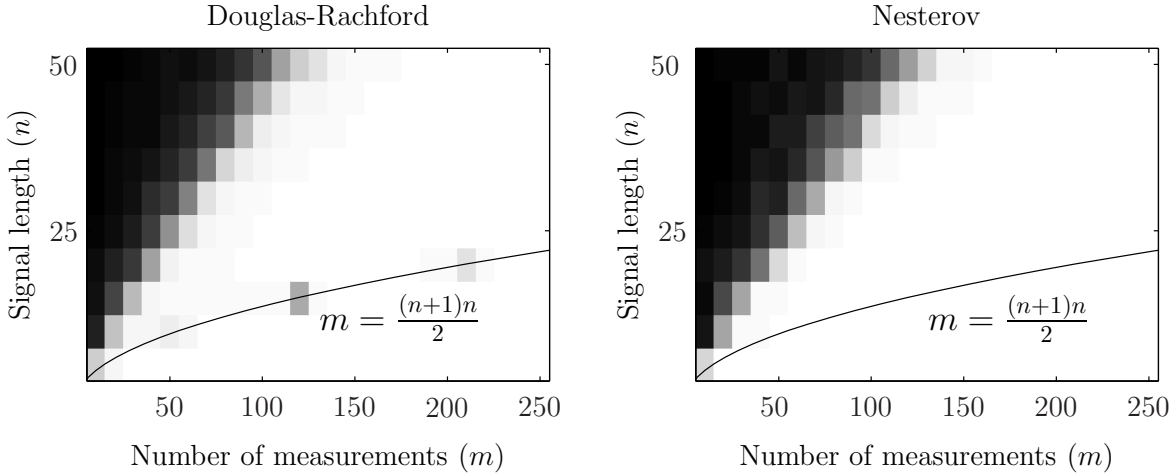
\begin{figure}[h]
	{\large
  \input{13045_fig2.tex}
  	}
  \caption{Recovery error for the Douglas-Rachford (DR) and Nesterov methods for the noisy optimizationless matrix recovery problem \eqref{noisy-axb-lifted} as a function of $n$ and $m$. In these plots, $\eps = 10^{-1}$.  For the Nesterov method, $\lambda = 0$. Black  represents an average recovery error of 100\%.   White represents zero average recovery error.  Each block corresponds to the average of  10 independent samples.  The solid curve depicts when there are the same number of measurements as degrees of freedom.  The number of measurements required for recovery appears to be roughly linear, as opposed to quadratic, in $n$.  The DR algorithm has large recovery errors near the curve where the number of measurements equals the number of degrees of freedom.  }
    \label{fig:phase-transition}
\end{figure}

Figure \ref{fig:single-recovery} shows recovery error versus iteration number under the Douglas-Rachford method, the Nesterov method for $\lambda = 0$ and the Nesterov method for $\lambda = 10^{-5}$.  For the Nesterov methods, we let the step size parameter be $\alpha = 10^{-4}$. 
For noisy data, convergence is initially linear until it tapers off around the noise level.  For noiseless data, convergence for feasibility problem is linear under both the Douglas-Rachford and Nesterov methods.  The Nesterov implementation of PhaseLift shows initial linear convergence until it tapers off.  Because any nonzero $\lambda$ allows for some data misfit in exchange for a smaller trace, the computed minimum is not $\X_0$ and the procedure converges to some nearby matrix.  The convergence rates of the Nesterov method could probably be improved by tuning the step-sizes in a more complicated way. Nonetheless, we observe that the Douglas-Rachford method exhibits a favorable convergence rate while requiring no parameter tuning.

We would like to remark that work subsequent to this paper shows that the number of measurements needed by the optimizationless feasibility problem is about the same as the number needed by PhaseLift  \cite{WDM2012}.  That is, the phase transition in Figure \ref{fig:phase-transition} occurs in about the same place for both problems.

\begin{figure}[h]
	{\large

  \input{13045_fig3.tex}
  	}
  \caption{The relative error versus iteration number for the noiseless and noisy matrix recovery problems,
  \eqref{exact-axb-lifted} and \eqref{noisy-axb-lifted}, under the Douglas-Rachford and Nesterov methods.  The left panel corresponds to the iteration \eqref{dr-first}--\eqref{dr-last}.  The middle panel corresponds to the iteration \eqref{nesterov-first}--\eqref{nesterov-last} in the optimizationless case, where $\lambda =0$.  The right panel corresponds to the iteration \eqref{nesterov-first}--\eqref{nesterov-last} for $\lambda=10^{-5}$. As expected, convergence is linear until it saturates due to noise.  In the Nesterov implementation of PhaseLift, convergence is linear until it saturates because the nonzero $\lambda$ makes the solution to the problem different than $\X_0$.}
  \label{fig:single-recovery}
\end{figure}
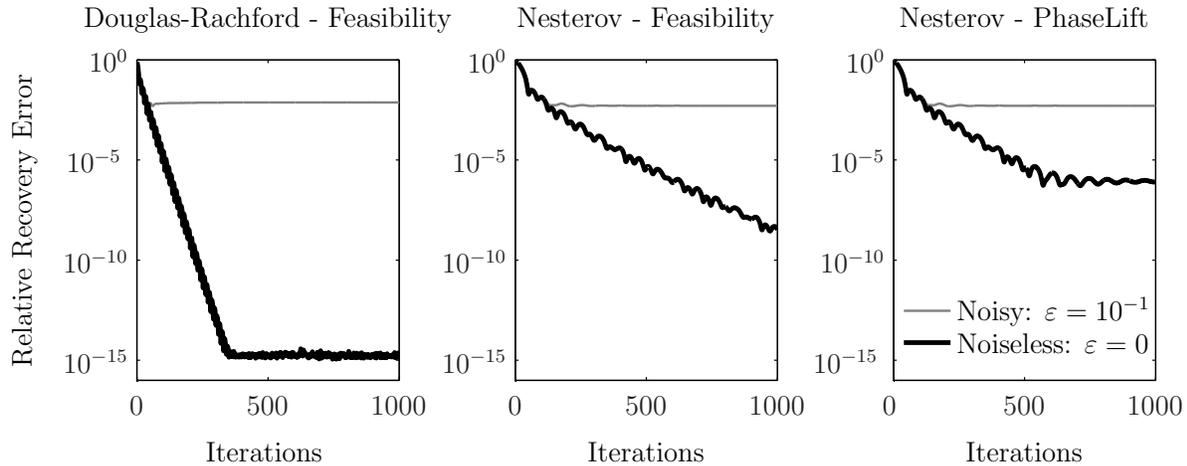

\end{document}

%% file: 13045_fig2.tex
%
%
\begin{psfrags}%
\psfragscanon%
%
\psfrag{s05}[t][t]{\color[rgb]{0,0,0}\setlength{\tabcolsep}{0pt}\begin{tabular}{c}Number of measurements ($m$)\end{tabular}}%
\psfrag{s07}[][]{\color[rgb]{0,0,0}\setlength{\tabcolsep}{0pt}\begin{tabular}{c}Signal length ($n$)\end{tabular}}%
\psfrag{s08}[][]{\color[rgb]{0,0,0}\setlength{\tabcolsep}{0pt}\begin{tabular}{c}Douglas-Rachford\end{tabular}}%
\psfrag{s09}[l][l]{\color[rgb]{0,0,0}\setlength{\tabcolsep}{0pt}\begin{tabular}{l}{\Large $m = \frac{(n+1)n}{2}$}\end{tabular}}%
\psfrag{s10}[t][t]{\color[rgb]{0,0,0}\setlength{\tabcolsep}{0pt}\begin{tabular}{c}Number of measurements ($m$)\end{tabular}}%
\psfrag{s12}[][]{\color[rgb]{0,0,0}\setlength{\tabcolsep}{0pt}\begin{tabular}{c}Signal length ($n$)\end{tabular}}%
\psfrag{s13}[][]{\color[rgb]{0,0,0}\setlength{\tabcolsep}{0pt}\begin{tabular}{c}Nesterov\end{tabular}}%
\psfrag{s14}[l][l]{\color[rgb]{0,0,0}\setlength{\tabcolsep}{0pt}\begin{tabular}{l}{\Large $m = \frac{(n+1)n}{2}$}\end{tabular}}%
%
\psfrag{x01}[t][t]{50}%
\psfrag{x02}[t][t]{100}%
\psfrag{x03}[t][t]{150}%
\psfrag{x04}[t][t]{200}%
\psfrag{x05}[t][t]{250}%
\psfrag{x06}[t][t]{50}%
\psfrag{x07}[t][t]{100}%
\psfrag{x08}[t][t]{150}%
\psfrag{x09}[t][t]{200}%
\psfrag{x10}[t][t]{250}%
%
\psfrag{v01}[r][r]{50}%
\psfrag{v02}[r][r]{25}%
\psfrag{v03}[r][r]{50}%
\psfrag{v04}[r][r]{25}%
%
\resizebox{16cm}{!}{\includegraphics{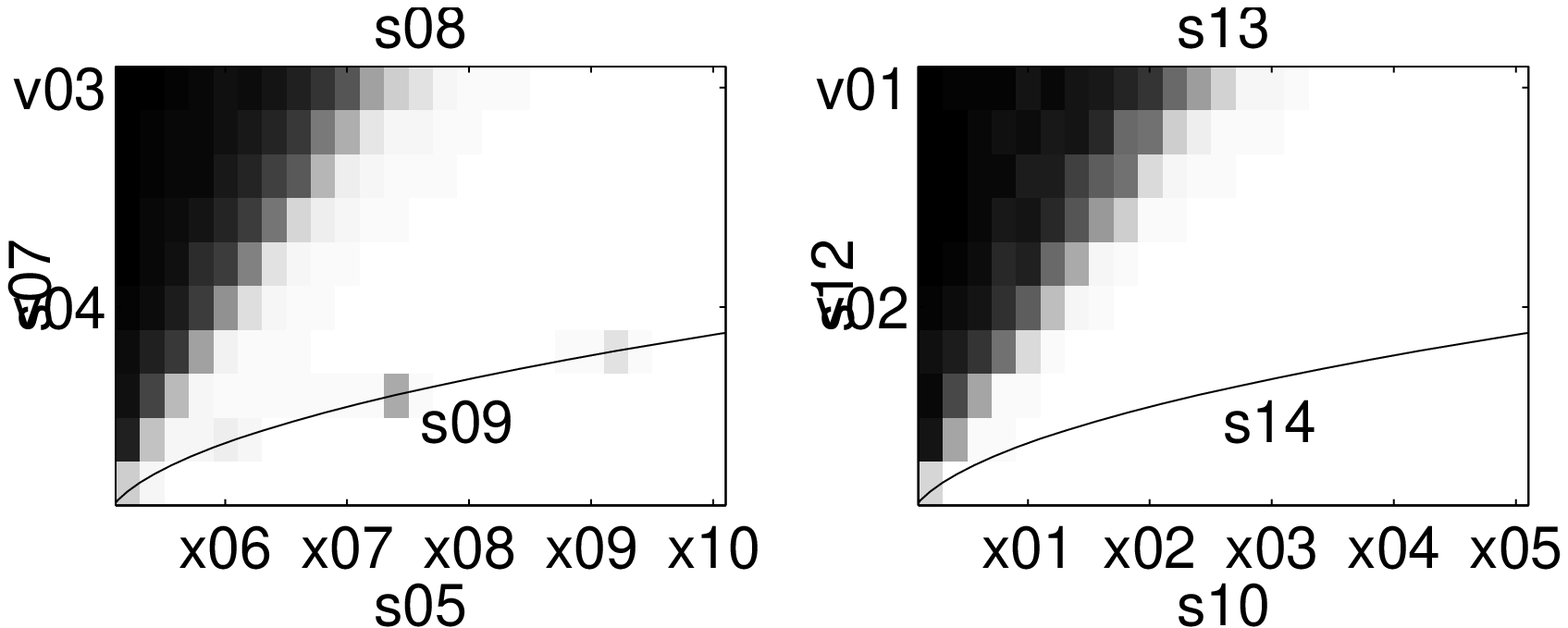}}%
\end{psfrags}%
%

%% file: 13045_fig3.tex
%
%
\begin{psfrags}%
\psfragscanon%
%
\psfrag{s10}[][]{\color[rgb]{0,0,0}\setlength{\tabcolsep}{0pt}\begin{tabular}{c}Douglas-Rachford - Feasibility\end{tabular}}%
\psfrag{s11}[t][t]{\color[rgb]{0,0,0}\setlength{\tabcolsep}{0pt}\begin{tabular}{c}Iterations\end{tabular}}%
\psfrag{s12}[b][b]{\color[rgb]{0,0,0}\setlength{\tabcolsep}{0pt}\begin{tabular}{c}Relative Recovery Error\end{tabular}}%
\psfrag{s13}[][]{\color[rgb]{0,0,0}\setlength{\tabcolsep}{0pt}\begin{tabular}{c}Nesterov - Feasibility\end{tabular}}%
\psfrag{s14}[t][t]{\color[rgb]{0,0,0}\setlength{\tabcolsep}{0pt}\begin{tabular}{c}Iterations\end{tabular}}%
\psfrag{s15}[][]{\color[rgb]{0,0,0}\setlength{\tabcolsep}{0pt}\begin{tabular}{c}Nesterov - PhaseLift\end{tabular}}%
\psfrag{s16}[t][t]{\color[rgb]{0,0,0}\setlength{\tabcolsep}{0pt}\begin{tabular}{c}Iterations\end{tabular}}%
\psfrag{s21}[][]{\color[rgb]{0,0,0}\setlength{\tabcolsep}{0pt}\begin{tabular}{c} \end{tabular}}%
\psfrag{s22}[][]{\color[rgb]{0,0,0}\setlength{\tabcolsep}{0pt}\begin{tabular}{c} \end{tabular}}%
\psfrag{s23}[l][l]{\color[rgb]{0,0,0}Noiseless: $\eps = 0$}%
\psfrag{s24}[l][l]{\color[rgb]{0,0,0}Noisy: $\eps = 10^{-1}$}%
\psfrag{s25}[l][l]{\color[rgb]{0,0,0}Noiseless: $\eps = 0$}%
%
\psfrag{x01}[t][t]{0}%
\psfrag{x02}[t][t]{500}%
\psfrag{x03}[t][t]{1000}%
\psfrag{x04}[t][t]{0}%
\psfrag{x05}[t][t]{500}%
\psfrag{x06}[t][t]{1000}%
\psfrag{x07}[t][t]{0}%
\psfrag{x08}[t][t]{500}%
\psfrag{x09}[t][t]{1000}%
%
\psfrag{v01}[r][r]{$10^{-15}$}%
\psfrag{v02}[r][r]{$10^{-10}$}%
\psfrag{v03}[r][r]{$10^{-5}$}%
\psfrag{v04}[r][r]{$10^{0}$}%
\psfrag{v05}[r][r]{$10^{-15}$}%
\psfrag{v06}[r][r]{$10^{-10}$}%
\psfrag{v07}[r][r]{$10^{-5}$}%
\psfrag{v08}[r][r]{$10^{0}$}%
\psfrag{v09}[r][r]{$10^{-15}$}%
\psfrag{v10}[r][r]{$10^{-10}$}%
\psfrag{v11}[r][r]{$10^{-5}$}%
\psfrag{v12}[r][r]{$10^{0}$}%
%
\resizebox{16cm}{!}{\includegraphics{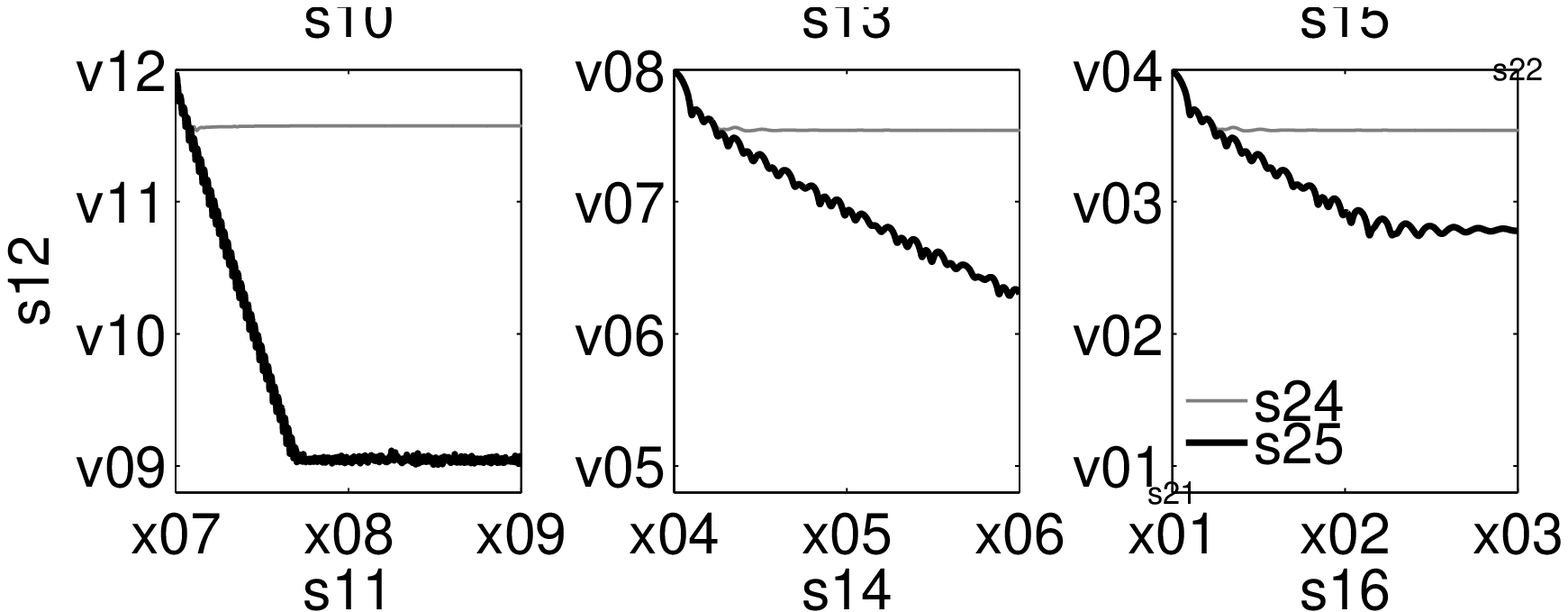}}%
\end{psfrags}%
%